\documentclass[11pt]{amsart}

\usepackage{amssymb}
\usepackage{todonotes}
\usepackage{tikz-cd}
\usepackage{hyperref}

\newtheorem{theorem}{Theorem}[section]
\newtheorem*{theorem*}{Theorem}
\newtheorem{lemma}[theorem]{Lemma}
\newtheorem*{lemma*}{Lemma}

\newtheorem*{proposition*}{Proposition}
\newtheorem{corollary}[theorem]{Corollary}
\newtheorem*{corollary*}{Corollary}

\theoremstyle{definition}
\newtheorem{definition}[theorem]{Definition}
\newtheorem*{definition*}{Definition}

\newtheorem*{example*}{Example}

\newtheorem*{ques*}{Question}

\newtheorem*{claim*}{Claim}
\newtheorem{remark}[theorem]{Remark}
\newtheorem*{remark*}{Remark}

\usepackage{bm}
\newcommand{\ii}{\item}
	\newcommand{\wt}{\widetilde}
\newcommand{\inv}{^{-1}}

\newcommand{\defeq}{:=}
\newcommand{\half}{\frac12}

\newcommand{\CC}{\mathbb C}

\newcommand{\QQ}{\mathbb Q}
\newcommand{\UU}{\mathbb U}
\newcommand{\RR}{\mathbb R}
\newcommand{\ZZ}{\mathbb Z}

\newcommand{\vv}[1]{\left\lVert #1 \right\rVert_{\nu}}

\DeclareMathOperator{\PGL}{PGL}
\newcommand{\Qc}{{\mathbb Q}^{\mathrm{c}}}
\newcommand{\kc}{k^{\mathrm{c}}}
\newcommand{\eps}{\varepsilon}
\newcommand{\BG}{{\mathbb G}_{\mathrm{m}}}
\newcommand{\vocab}{\textit}
\newcommand{\ol}{\overline}
\usepackage{mathrsfs}
\newcommand{\LL}{\mathscr{L}}

\usepackage{tikz}
\usetikzlibrary{shapes,arrows,arrows.meta}

\makeatletter
\newcommand*{\house}[1]{%
  \mathord{%
    \mathpalette\@house{#1}%
  }%
}
\newcommand*{\@house}[2]{%
  \dimen@=\fontdimen8 %
      \ifx#1\scriptscriptstyle\scriptscriptfont
      \else\ifx#1\scriptstyle\scriptfont
      \else\textfont\fi\fi
      3 %
  \sbox0{%
    $#1%
      \vrule width\dimen@\relax
      \overline{%
        \kern2\dimen@
        \begingroup % to keep changes of \dimen@ in #2 local
          #2%
        \endgroup
        \kern2\dimen@
      }%
      \vrule width\dimen@\relax
      \mathsurround=1.5\dimen@ % outside margin
    $%
  }%
  \ht0=\dimexpr\ht0-\dimen@\relax
  \dp0=\dimexpr\dp0+2\dimen@\relax
  \vbox{%
    \kern\dimen@ % reinsert previously removed space
    \copy0 %
  }%
}

\title[Avoiding integers in orbits over cyclotomic closures]%
{Avoiding algebraic integers of bounded house
in orbits of rational functions over cyclotomic closures}
\date{May 18, 2018}

\usepackage{amsaddr}

\author{Evan Chen}
\address{Department of Mathematics, Massachusetts Institute of Technology}
\email{evanchen@mit.edu}


\begin{document}
\maketitle

\begin{abstract}
	Let $k$ be a number field with cyclotomic closure $k^{\mathrm{c}}$,
	and let $h \in k^{\mathrm{c}}(x)$.
	For $A \ge 1$ a real number, we show that
	\[ \{ \alpha \in k^{\mathrm{c}} :
		h(\alpha) \in \overline{\mathbb Z}
		\text{ has house at most } A \} \]
	is finite for many $h$.
	We also show that for many such $h$
	the same result holds if $h(\alpha)$ is replaced
	by orbits $h(h(\cdots h(\alpha)))$.
	This generalizes a result proved by Ostafe
	that concerns avoiding roots of unity, which is the case $A=1$.
\end{abstract}

\section{Introduction}
\subsection{Rational functions and set avoidance}
We begin with the following general definition.
\begin{definition}
	Let $F$ be a subfield of $\CC$, and $P$ a subset of $\CC$.
	Let $h \in F(x)$ be a rational function,
	and let $h^n$ denote the function composition of
	$h$ applied $n$ times ($n=0,1,2,\dots$).
	\begin{itemize}
		\ii We say that $h$ is \vocab{$P$-avoiding}
		%\todo{suggestions for better terms welcome}
		(over $F$) if
		\[ \#\left\{ \alpha \in F \mid h(\alpha) \in P \right\} < \infty. \]
		\ii We say that $h$ is \vocab{strongly $P$-avoiding} (over $F$) if
		\[ \#\left\{ \alpha \in F
			\mid h^n(\alpha) \in P \text{ for some $n \ge 1$} \right\}
			< \infty. \]
	\end{itemize}
\end{definition}

Let $\UU \subseteq \CC$ denote the set of roots of unity
and let $k$ be a number field.
We will denote its cyclotomic closure $k(\UU)$ by $\kc$.
This paper will concern avoidance over $\kc$.

We say a rational function $h(x) \in \kc(x)$
is \vocab{special} if $h$ is conjugate,
with respect to a M\"{o}bius transformation (i.e.\ via $\PGL_2(\kc)$),
to either $\pm x^d$ or the Chebyshev polynomial $T_d(x)$
which is uniquely determined by the equation
$T_d(\half(t + t\inv)) = \half(t^d + t^{-d})$.

The question of $\UU$-avoidance and strong $\UU$-avoidance
has been examined by Dvornicich and Zannier.
For example, as a consequence of \cite[Corollary 1]{ref:zannier},
we have the following result.
\begin{theorem*}
	[From {\cite[Corollary 1]{ref:zannier}}]
	Let $h = p/q \in \kc(x)$, where $p,q \in \kc[x]$.
	Assume that $p(x) - y^m q(x)$ is irreducible over $\kc$
	for all positive integers $m \le \max(\deg p, \deg q)$.
	Then $h$ is $\UU$-avoiding over $\kc$.
\end{theorem*}
Ostafe \cite{ref:ostafe} proved
the following result for strong $\UU$-avoidance.
\begin{theorem*}
	[{\cite[Theorem 1.2]{ref:ostafe}}]
	Let $h = p/q \in k(x)$, where $p,q \in k[x]$.
	Assume $h$ is $\UU$-avoiding over $\kc$,
	and $\deg p > \deg q + 1$.
	Assume also that $\max(\deg p, \deg q) \ge 2$
	and $p(x) - y^m q(x)$ as a polynomial in $x$
	does not have a root in $\kc(y)$
	for all positive integers $m \le \deg(p)$.
	Then $h$ is strongly $\UU$-avoiding unless $h$ is special.
\end{theorem*}

In this paper we investigate a generalization of these results
proposed by Ostafe (see \cite[\S4]{ref:ostafe}).
In order to state it, we need to define the following.
\begin{definition}
	The \vocab{house} of an algebraic number $\alpha$,
	denoted $\house{\alpha}$, is the maximum value of
	$\left\lvert \beta \right\rvert$ across the
	$\QQ$-Galois conjugates $\beta$ of $\alpha$.

	For $A \ge 1$ a real number,
	let $P_A$ denote the set of algebraic integers $\alpha$
	which have house at most $A$.
\end{definition}
For example every algebraic integer has house at least $1$,
and by Kronecker's theorem (the main result of \cite{ref:kro1},
see also \cite{ref:kro2}) we have $P_1 = \UU$.

We answer the following question.
\begin{ques*}
	For $A \ge 1$ and $h \in \kc(x)$,
	under what conditions can one show that $h$
	is (strongly) $P_A$-avoiding?
\end{ques*}

\subsection{Summary of results}
The \vocab{degree} of a nonconstant rational function
$h$ with coefficients in some field $F$ is defined to be $[F(x) : F(h(x))]$.
%\todo{I think degree is the standard word for this, unfortunately?}
Consequently, note that $\deg(h_1 \circ h_2) = \deg h_1 \deg h_2$.
If $h$ is written as a quotient of relatively prime polynomials $p/q$,
then $\deg h = \max(\deg p, \deg q)$.

Our results on $P_A$-avoidance can be summarized as follows.
\begin{theorem}
	\label{thm:salespitch}
	Let $k$ be a number field, $A \ge 1$ and $\eps > 0$.
	Let $h \in \kc(x)$ be a rational function.
	\begin{itemize}
		\ii Then $h$ is $P_A$-avoiding
		unless there exists $S \in \kc(x)$
		such that $h(S(x))$ equals a Laurent polynomial
		with $d$ terms, where \[ d \ll_{k,\eps} A^{2+\eps}. \]
		\ii If $\deg h \gg_{k,A} 1$,
		then we can also assume $\deg S \le 2$.
	\end{itemize}
\end{theorem}
This theorem has an effective and more explicit form
given as Theorem~\ref{thm:big} and Theorem~\ref{thm:most}.

A corollary of Theorem~\ref{thm:salespitch} is the following.
\begin{corollary}
	\label{cor:amusing}
	Let $k$ be a number field and $A \ge 1$.
	If $h$ has more than two poles, then $h$ is $P_A$-avoiding.
\end{corollary}

Using this result, we will deduce the following
generalization of a result of Ostafe \cite[Theorem 1.2]{ref:ostafe},
and give a simple proof using Theorem~\ref{thm:big}.
\begin{theorem}
	\label{thm:strongavoid}
	Let $h = p/q \in k(x)$ , where $p,q \in k[x]$.
	Let $A \ge 1$.
	Assume $h$ is $P_A$-avoiding over $\kc$,
	and $\deg p > \deg q + 1$.
	Then $h$ is strongly $P_A$-avoiding unless $h$ is special.
\end{theorem}

\subsection{Outline}
The rest of the paper is structured as follows.
In Section~\ref{sec:loxton}, we state the Loxton theorem,
namely Theorem~\ref{thm:loxton},
and use this to give a more precise version of
Theorem~\ref{thm:salespitch} as
Theorem~\ref{thm:big} and Theorem~\ref{thm:most}.
In Section~\ref{sec:nukes}, we introduce several
auxiliary results which will be used in our proofs.

In Section~\ref{sec:usenuke} we prove Theorem~\ref{thm:big}
and Theorem~\ref{thm:most}, as well as Corollary~\ref{cor:amusing};
these are our results on $P_A$-avoidance.
Finally, Section~\label{sec:orbit} gives the proof of
Theorem~\ref{thm:strongavoid},
which is our result on strong $P_A$-avoidance.
% Appendix~\ref{sec:flowchart} summarizes
% the dependencies among the various parts of this paper.

\subsection*{Acknowledgments}
This research was funded by NSF grant 1358659
and NSA grant H98230-16-1-0026
as part of the 2016 Duluth Research Experience for Undergraduates (REU).

The author thanks Joe Gallian for supervising the research,
and for suggesting the problem.

The author also thanks Joe Gallian, Aaron Landesman, and David Yang
for their close readings of early drafts of the paper,
as well as Benjamin Gunby and Levent Alpoge for helpful discussions.
Finally the author is deeply grateful to the anonymous referees
for extensive comments and corrections.

\section{Full statement of results on $P_A$-avoidance}
\label{sec:loxton}
In order to recall the full version of Theorem~\ref{thm:salespitch},
we firstly need to state the following extension
of a theorem of Loxton \cite[Theorem 1]{ref:loxton}.
\begin{theorem}
	[Loxton theorem, {\cite[Theorem L]{ref:zannier}}]
	\label{thm:loxton}
	There exists a function $\LL \colon \RR_+ \to \RR_+$
	with the following property.
	For every number field $k$,
	we can fix a real number $B > 0$ and a finite subset $E \subseteq k$
	of cardinality at most $[k:\QQ]$ so that
	every algebraic integer $\alpha$ in $\kc$ can be written as
	\[ \sum_{i=1}^d e_i \xi_i \]
	where $e_i \in E$, $\xi_i \in \UU$,
	and $d \le \LL(B \cdot \house{\alpha})$.
\end{theorem}
In light of this, it will be convenient to make the following definition.
\begin{definition}
	For every number field $k$
	we fix a pair $(B,E)$ (depending only on $k$) as above.
	We will call this the \vocab{Loxton pair} for $k$.
	The \vocab{Loxton function} $\LL$ will also remain fixed through the paper.
\end{definition}
\begin{remark}
	\label{rem:loxdetails}
	The exact nature of $\LL$ is not important for our purposes.
	However, it is possible to choose
	$\LL(x) = O_{\varepsilon}(x^{2+\varepsilon})$.
	Moreover, in the case $k = \QQ$ one can select $E = \{1\}$.
	See \cite{ref:loxton} for more details.
\end{remark}

\begin{definition}
	Let $h \in \kc(x)$ and fix $(B,E)$ a Loxton pair for $k$.
	Suppose that there exist a nonconstant $S \in \kc(x)$,
	integers $n_i$, roots of unity $\beta_i \in \UU$,
	and $e_i \in E$ which satisfy
	\[ \sum_{i=1}^d \beta_i e_i x^{n_i} = h(S(x)). \]
	In this case, we call the rational function
	$\sum \beta_i e_i x^{n_i}$ a \vocab{witness} for $h$.

	If $A \ge 1$ is a real number,
	the witness is called \vocab{$A$-short} if $d \le \LL(AB)$.
\end{definition}

Observe that, if there exists a witness for $h$,
then $h$ is seen to not be $P_A$-avoiding for sufficiently large $A$,
by simply selecting $x \in \UU$.
We will prove the following result.

\begin{theorem}
	\label{thm:big}
	Let $h(x) \in \kc(x)$ be nonconstant, and $A \ge 1$.
	Then $h$ is $P_A$-avoiding
	unless there exists an $A$-short witness for $h$.
\end{theorem}

According to Remark~\ref{rem:loxdetails} above,
the case $k = \QQ$ has a particularly nice phrasing.
\begin{corollary}
	Let $h(x) \in \Qc(x)$ be nonconstant and $A \ge 1$.
	Then $h$ is $P_A$-avoiding unless there exists $S \in \Qc(x)$
	such that $h(S(x))$ is equal to a Laurent polynomial
	$p \in \ZZ[\UU][x,x\inv]$
	with $|p(1)| \ll_{\varepsilon} A^{2+\varepsilon}$.
\end{corollary}

As stated, these results do not give any bound
on the size of the degree of a witness.
However, the following theorem shows that
``most'' of $h(x) \in \kc(x)$ are in fact $P_A$-avoiding.
\begin{theorem}
	\label{thm:most}
	Let $k$ be a number field with Loxton pair $(B,E)$.
	Let $A \ge 1$ and let $h(x) \in \kc(x)$ be nonconstant.
	Suppose that
	\begin{itemize}
		\ii $\deg h > 2016 \cdot 5^{\LL(AB)+1}$, or
		\ii $h$ is a polynomial and $\deg h > (2\LL(AB)+1)^2$.
	\end{itemize}
	Then $h$ is $P_A$ avoiding
	unless it has an $A$-short witness $h(S(x))$ for which
	$\deg S \le 2$.
\end{theorem}
\begin{remark}
	In fact, if $h \in \kc[x]$ is a polynomial which is not $P_A$-avoiding,
	one can find an $A$-short witness
	of the form $h(ax + b + cx\inv)$ for some $a,b,c \in \kc$
	(see Theorem~\ref{thm:terms}).
\end{remark}
\begin{remark}
	The constants involved in Theorem~\ref{thm:most}
	come from Fuchs-Zannier \cite{ref:terms},
	reproduced in the next section as Theorem~\ref{thm:terms}.
\end{remark}

\section{Background}
\label{sec:nukes}
To prove the main result, we will need other auxiliary results,
which we collect in this section.

\subsection{Tools from arithmetic geometry}
In what follows, fix $k$ a number field,
and $\BG = \operatorname{Spec} k[x, x\inv]$ as usual.
By a \vocab{torsion coset} of $\BG^d$,
we mean a translate $\beta \cdot T$
of a subtorus $T$ (i.e.\ a connected algebraic group)
by a torsion point $\beta$ of $\BG^d$.

\begin{theorem}
	[{\cite[Torsion Points Theorem]{ref:zannier}}]
	\label{thm:torsionpts}
	Let $V$ be an algebraic subvariety of $\BG^d$ defined over $\ol{\QQ}$.
	Then the Zariski closure of the set of torsion points in $V$
	is a finite union of torsion cosets of $\BG^d$.
\end{theorem}

We also use a special case of \cite[Theorem 1]{ref:zannier}.
\begin{theorem}
	\label{thm:zmain}
	Let $k$ be a number field.
	Let $V/k$ be an affine variety irreducible over $\kc$
	and let \[ \pi : V \to \BG^r \] be a morphism of finite degree,
	defined over $k$.
	Assume the set of torsion points of $\pi(V(\kc))$
	is Zariski-dense in $\BG^r$.

	Then, there exists an isogeny $\mu : \BG^r \to \BG^r$
	and a birational map $\rho : \BG^r \dashrightarrow V$,
	both defined over $\kc$, such that the diagram
	\begin{center}
	\begin{tikzcd}
		\BG^r \rar[dashed]{\rho} \drar[swap]{\mu}
			& V \dar{\pi} \\
		& \BG^r
	\end{tikzcd}
	\end{center}
	commutes (over $\kc$).
\end{theorem}
\begin{proof}
	We define the set
	\[ J = \{ \eta \in V(\kc) : \pi(\eta) \text{ is a torsion point of } \BG^r \}. \]
	Thus $\pi(J)$ consists exactly of all torsion points of $\pi(V(\kc))$,
	so it is Zariski-dense by hypothesis.
	Since $\pi$ is of finite degree,
	it follows that $J$ is Zariski-dense in $V$ as well.
	Then we can apply \cite[Theorem 1]{ref:zannier},
	where the torsion coset $T$ in question is the entire $\BG^r$.
\end{proof}

\subsection{Results on compositions of rational functions}
We recall the following results of Fuchs and Zannier \cite{ref:terms}.
These results hold in much more generality if $\kc$
is replaced by any field of characteristic zero,
but we will not need that generality for our purposes.
\begin{theorem}
	[{\cite[Main Theorem and Theorem 2]{ref:terms}}]
	\label{thm:terms}
	Let $p, q, h \in \kc(x)$ be rational functions with $p = h \circ q$,
	Denote by $\ell$ the sum of the number of terms
	in the numerator and denominator of $p$.
	\begin{itemize}
		\ii Assume $q$ is not of the shape
		$\lambda(ax^n+bx^{-n})$ for $a,b \in \kc$,
		$\lambda \in \PGL_2(\kc)$, $n \in \ZZ_{>0}$.
		Then, \[ \deg h \le 2016 \cdot 5^\ell. \]
		\ii Suppose $p \in \kc[x,x\inv] \setminus \kc[x]$
		is a Laurent polynomial with $\ell$ nonconstant
		terms for some $\ell \ge 0$.
		Suppose moreover that $h \in \kc[x]$ is a polynomial
		and $q \in \kc[x, x\inv]$, where $q(x)$ is not of the shape
		$ax^n + b + cx^{-n}$ for $a,b,c \in \kc$, $n \in \ZZ_{>0}$.
		Then, \[ \deg h \le 2(2\ell-1)(\ell-1). \]
	\end{itemize}
\end{theorem}
\begin{corollary}
	[{\cite[Corollary on pg.\ 177]{ref:terms}}]
	\label{thm:specialterms}
	Let $q \in \kc(x)$ be non-constant,
	and $h \in \kc(x)$ with $\deg h \ge 3$ not special.
	Then for any integer $n \ge 3$,
	the sum of the number of terms in the numerator and denominator
	of the rational function $h^n \circ q$ is at least
	\[ \log_5 \left( \frac{(\deg h)^{n-2}}{2016} \right). \]
\end{corollary}

\subsection{Estimates on sizes of orbits}
We will use the following result,
which is based on \cite[\S1.3]{ref:ostafe}.
\begin{lemma}
	\label{lem:ostafe}
	Let $k$ be a number field and let $h = p/q \in k(x)$ be a rational function.
	Assume $\deg p > \deg q + 1$.

	Then, there exist a real number $T > 0$ and an integer $D$
	(depending only on $h$) with the following properties.
	For any algebraic number $\alpha$,
	\begin{itemize}
		\ii If $\house{h^n(\alpha)} \le A$ for some $n \ge 1$,
		then \[ \house{h^j(\alpha)} \le \max(T,A)
			\quad\text{ for } j=0,\dots,n-1. \]
		\ii If $h^n(\alpha)$ is an algebraic integer
		for some $n \ge 1$,
		then $Dh^j(\alpha)$ is an algebraic integer
		for $j=0,1,\dots,n-1$.
	\end{itemize}
\end{lemma}
\begin{proof}
	Suppose that $h^n(\alpha) = \gamma$.

	First, since $\deg p - \deg q \neq 1$ we can pick $0 \neq c \in \ol{\QQ}$
	(depending only on $h$) such that
	\[ h(x) = c \inv \cdot \wt h (cx) \]
	and moreover $\wt h$ is ``monic'' in the sense that $\wt h = \wt p / \wt q$ and
	\begin{align*}
		\wt p(x) &= x^d + a_{d-1} x^{d-1} + \dots + a_0 \\
		\wt q(x) &= x^e + b_{e-1} x^{e-1} + \dots + b_0.
	\end{align*}
	(It is possible that $c \notin k$;
	in this case we enlarge $k$ to contain $c$).
	Now, for any $j=0,\dots,n$ we have
	\[ h^j(x) = c\inv \cdot {\wt h}^j(cx). \]
	In particular, $\wt h^j(c\alpha) = c\gamma$.

	The first part now follows from applying \cite[Corollary 2.7]{ref:ostafe},
	to $cA$, $c\alpha$ and $\wt h$,
	using the condition $\deg p - \deg q > 1$.

	We proceed to the second part.
	Assume $\gamma$ is an algebraic integer.
	Note that by replacing the value of $n$,
	it suffices just to show that $D\alpha$ is an algebraic integer
	for some integer $D$ depending only on $h$.

	Let $\nu$ be an arbitrary finite place of $k$.
	Then \cite[Corollary 2.5]{ref:ostafe} implies that
	if $\vv{c\alpha} > \max\{1, \vv{a_i}, \vv{b_i}\}$
	then the sequence \[ \vv{\wt h^j(c\alpha)} \qquad \text{for } j=0,1,2,\dots \]
	is strictly increasing.
	Thus, in particular we must have
	\[ \vv{c\alpha} \le
		\max\left( 1, \vv{a_i}, \vv{b_i}, \vv{c\gamma} \right) \]
	or else we contradict the fact that $\wt h^j(c\alpha) = c\gamma$.

	Now, let $D$ be an integer for which
	$Dc\inv $, $Dc\inv a_i$, $Dc\inv b_i$ are all algebraic integers.
	Multiplying the previous inequality by $Dc\inv$, we obtain
	\begin{align*}
		\vv{D\alpha} &\le
		\max\left( \vv{Dc\inv}, \vv{Dc\inv a_i}, \vv{Dc\inv b_i},
		\vv{D \gamma} \right) \\
		&\le 1.
	\end{align*}
	Since this is true for every finite place $\nu$,
	it follows that $D\alpha$ is an integer.
	Moreover, since $D$ depends only on $c$, $a_i$, $b_i$
	and not on $\gamma$, it follows that $D$ depends only on $h$,
	which proves our assertion.
\end{proof}

\section{Proof of results on $P_A$-avoidance}
\label{sec:usenuke}

\begin{proof}[Proof of Theorem~\ref{thm:big}]
	Assume $h$ is not $P_A$-avoiding, so $h(\kc)$ contains infinitely elements of $P_A$.
	By Theorem~\ref{thm:loxton} and the pigeonhole principle,
	we can fix $d \le \LL(AB)$ and $e_i \in E$ such that
	there exist infinitely many elements $y \in \kc$
	and $\xi_1, \dots, \xi_d \in \UU$ satisfying
	\[ h(y) = \sum_{i=1}^d e_i \xi_i. \]

	\bigskip

	Take $\BG^{d+1}$ equipped with coordinates $(x_1, \dots, x_d, y)$.
	Letting $h = p/q$ for $p, q \in \kc[x]$,
	consider the subvariety
	\[ V \subseteq \BG^{d+1} \]
	defined by the equation
	\[ p(y) = q(y) \sum_{i=1}^d e_i x_i. \]
	Moreover, let $\UU_d$ denote the set of torison points of $\BG^d$
	and let $\Pi : V \to \BG^d$ be the projection onto the first $d$ coordinates.
	We now consider the following iterative procedure.
	Initially, let 
	\[ W_0 = V, \quad \bm \beta_0 = \bm 1 \in \BG^d,
		\quad \text{and } T_0 = \BG^d \]
	so the torsion coset $\bm \beta_0 T_0$ is all of $\BG^d$.
	So we have $\Pi(W_0) \subseteq \bm \beta_0 T_0$ and
	$\#(\Pi(W_0) \cap \UU_d) = \infty$.
	Then we recursively perform the following procedure for $i=0,1,2,\dots$.
	\begin{itemize}
		\ii Consider the infinite set
		$\bm \beta_i\inv \Pi(W_i) \cap \UU_d \subseteq T_i$.
		By Theorem~\ref{thm:torsionpts} applied to the subvariety $T_i$,
		its Zariski closure consists of finitely many torsion cosets.
		%\todo{referee expressed some concerns about this step at i=0,
		% but I think it's correct}
		Hence by pigeonhole principle, we may pick a particular torsion coset,
		say $\bm\beta' T_{i+1}$, containing infinitely many elements of $\UU_d$.
		Now set $\bm\beta_{i+1} = \bm\beta_i \bm\beta'$.
		Then we conclude that $\bm\beta_{i+1} T_{i+1}$
		is the closure of some infinite subset of $\Pi(W_i) \cap \UU_d$.

		\ii Now consider the preimage $\Pi\inv(\bm \beta_{i+1} T_{i+1})$,
		which is a closed subvariety of $W_i$.
		Then by pigeonhole principle,
		we can set $W_{i+1}$ to be any irreducible component of $W_i$
		such that $\# (\Pi(W_{i+1}) \cap \UU_d) = \infty$.
		Of course by construction $\Pi(W_{i+1}) \subseteq \bm\beta_{i+1} T_{i+1}$.
	\end{itemize}
	From this we have constructed
	\[ V = W_0 \supseteq W_1 \supseteq \cdots \]
	a decreasing sequence of subvarieties of $V$,
	with $W_i$ irreducible for $i \ge 1$.
	For dimension reasons, this sequence must eventually stabilize.
	Thus the torsion coset $\bm\beta_i T_i$ stabilizes too.
	So we conclude there exists
	\begin{itemize}
		\ii an \emph{irreducible} affine subvariety $W \subseteq V$,
		\ii a particular torsion coset $\bm\beta T \subseteq \BG^d$,
		where $\bm\beta = (\beta_1, \dots, \beta_d) \in \UU^d$
		and $T$ is a torus, and
		\ii $Z \defeq \Pi(W) \cap \UU_d$ a set of torsion points of $\BG^d$
	\end{itemize}
	such that
	\[ \Pi(W) \subseteq \bm\beta T,
		\qquad \ol Z = \bm\beta T,
		\qquad\text{and}\quad
		\#Z = \infty.  \]
	(In the case $V$ is already an irreducible subvariety,
	then $W=V$, the torsion coset $\bm\beta T$ is exactly $\BG^d$,
	and $Z=\UU_d$.
	On the other hand if $V$ is not irreducible then
	the $W_i$ start to decrease after the first step.)

	Let $r \defeq \dim T$; note that $r \ge 1$
	since $T$ contains the infinite set $Z$.

	\bigskip

	We now wish to apply Theorem~\ref{thm:zmain}.
	Consider the composed map $\pi : W \to \BG^r$
	defined by taking $\varphi$ as below:
	\begin{center}
	\begin{tikzcd}
		W \rar{\varphi} & T \rar[two heads]{\psi}[swap]{\simeq} & \BG^r \\
		(x_1, \dots, x_d, y) \rar[mapsto]
			& (\beta_1\inv x_1, \dots, \beta_d\inv x_d). &
	\end{tikzcd}
	\end{center}
	From the fact that $\ol Z = \bm\beta \cdot T$,
	we conclude that the set of torsion points in $\pi(W)$
	is Zariski dense in $\BG^r$.
	Applying Theorem~\ref{thm:zmain},
	there exist an isogeny $\mu : \BG^r \to \BG^r$
	and a birational map $\rho : \BG^r \dashrightarrow W$
	such that the diagram
	\begin{center}
	\begin{tikzcd}
		\BG^r \rar[dashed]{\rho} \drar[swap]{\mu}
			& W \dar{\pi} \rar{\varphi} & T \\
		& \BG^r \urar[swap]{\psi\inv} &
	\end{tikzcd}
	\end{center}
	commutes.

	Assume
	\begin{align*}
		\rho(\bm x) &= (R_1(\bm x), \dots, R_d(\bm x), R(\bm x)) \\
		\intertext{for rational functions $R_1, \dots, R_d, R$ (here $\bm x \in \BG^r$);
		then} \varphi(\rho(\bm x)) &=
		\left( \beta_1\inv R_1(\bm x), \dots, \beta_d\inv R_d(\bm x), R(\bm x) \right).
	\end{align*}
	Now, the right-hand side of $\varphi \circ \rho = \psi\inv \circ \mu$
	is the composition of an isogeny and an isomorphism,
	thus (for instance by \cite[Proposition 3.2.17]{ref:tori}),
	we recover that $R_i(\bm x) = \beta_i \bm x^{\bm v_i}$
	for some vectors $\bm v_i \in \ZZ^r$
	which are linearly independent (and in particular nonzero).

	Thus
	\[ \rho(\bm x) = (\beta_1 \bm x^{\bm v_1}, \dots, \beta_d \bm x^{\bm v_d},
		R(\bm x)) \]
	and we obtain an identity
	\[ h(R(\bm x)) = \sum_{i=1}^d e_i \cdot \beta_i \bm{x}^{\bm v_i}.  \]
	Since the $\bm{v_i}$ are independent,
	it follows that one can specialize $\bm x$ to a
	choice of the form $\bm x = (x^{c_1}, \dots, x^{c_r})$
	for some integers $c_i \in \ZZ$
	so that the terms $\bm{x}^{\bm{v_i}}$ are pairwise distinct.
	Thus we finally obtain
	\[ h(S(x)) = \sum_{i=1}^d \beta_i e_i x^{n_i} \]
	where $S$ is a rational function
	(defined by $S(x) \defeq R(x^{n_r}, \dots, x^{c_r})$),
	and the right-hand side is nonconstant in $x$.
	This is the desired $A$-short witness.
\end{proof}

\begin{proof}
	[Proof of Theorem~\ref{thm:most}]
	First suppose $h(x) \in \kc(x)$.
	Then by Theorem~\ref{thm:big}, $h$ is $P_A$-avoiding
	unless we have an identity
	\[ h(S(x)) = \sum_{i=1}^d \beta_i e_i x^{n_i} \]
	where the right-hand side
	has at most $d \le \LL(A \cdot B)$ terms.

	First assume $S = \mu(ax^n+bx^{-n})$ for some $\mu \in \PGL_2(k)$.
	Set now $\wt S = \mu(ax+bx\inv)$, $\deg \wt S = 2$.
	We now see that
	\[ h(\wt S(x)) \]
	is an $A$-short witness, establishing the theorem.

	Otherwise Theorem~\ref{thm:terms} applies with $\ell = d+1$,
	and we deduce that \[ \deg h \le 2016 \cdot 5^{d+1} \]
	which contradicts the first hypothesis of Theorem~\ref{thm:most}.
	This implies one direction.

	In the case $h \in \kc[x]$,
	we repeat the same argument,
	applying the second part of Theorem~\ref{thm:terms}.
	(That $S$ is a Laurent polynomial follows
	from the fact that it cannot have any nonzero poles,
	in light of the right-hand side having the same property.)
\end{proof}

\begin{proof}
	[Proof of Corollary~\ref{cor:amusing}]
	Suppose by contradiction $h$ is not $P_A$-avoiding;
	then by Theorem~\ref{thm:big} there is an $A$-short witness
	and we may write
	\[ h(S(x)) = \sum_i \beta_i e_i x^{n_i}. \]
	View this as an identity of rational functions in $\CC(x)$.

	On the one hand, since $S \in \CC(x)$ is a
	nonconstant rational function,
	its range in $\CC$ omits at most one point of $\CC$.
	Since $h$ has at least three poles,
	it follows that there is an $x_0 \neq 0$ such that $S(x_0)$
	is a pole of $h$.

	On the other hand, the only possible pole of the right-hand side
	is $x = 0$, which is the desired contradiction.
\end{proof}

\section{Proof of results on strong $P_A$-avoidance}
\label{sec:orbits}
%We now prove Theorem~\ref{thm:strongavoid},
%which we restate here for convenience of the reader.
%\begin{theorem*}
%	Let $h = p/q \in k(x)$ , where $p,q \in k[x]$.
%	Let $A \ge 1$.
%	Assume $h$ is $P_A$-avoiding over $\kc$,
%	and $\deg p > \deg q + 1$.
%	Then $h$ is strongly $P_A$-avoiding unless $h$ is special.
%\end{theorem*}

\begin{proof}
	[Proof of Theorem~\ref{thm:strongavoid}]
	Since $h$ is given to be $P_A$-avoiding,
	it suffices to show that for a given $\gamma \in P_A$,
	there are only finitely many $\alpha \in \kc$
	such that $h^n(\alpha) = \gamma$ for some $n \ge 1$.

	Assume by contradiction there are infinitely many
	pairs $(\alpha, n)$ such that $h^n(\alpha) = \gamma$.
	Select $T > 0$ and $D \in \ZZ$ by Lemma~\ref{lem:ostafe},
	and let \[ C \defeq D\max(T, A). \]
	We make the following claim.
	\begin{claim*}
		For any integer $N$,
		$D \cdot h^N(x)$ is not weakly $P_C$-avoiding.
	\end{claim*}
	To see this, discard the finitely many pairs with $n \le N$,
	and consider only those with $n > N$.
	Then by applying Lemma~\ref{lem:ostafe}
	to such pairs $(\alpha, n)$ with $n > N$,
	there are infinitely many $\alpha$
	such that $D \cdot h^N(\alpha)$ is an algebraic integer;
	moreover, the house of $D \cdot h^N(\alpha)$
	is at most $D \cdot \max(T,A) = C$, giving the claim.

	Consequently, by Theorem~\ref{thm:big}
	for every integer $N$ there exists
	a $C$-short witness.
	In other words, for all $N \ge 1$ there
	exists $S \in \kc(x)$ such that
	\[ D \cdot h^N(S(x)) = \sum_{i=1}^d \beta_i e_i x^{n_i} \]
	where $d \le \LL(BC) = \LL(BD\max(T,A))$.

	By hypothesis, $\deg h \ge 2$. Assume that $\deg h \ge 3$.
	Since we are given that $h$ is special,
	by Corollary~\ref{thm:specialterms},
	$h^N$ has at least $\log \left( \frac{(\deg h)^{N-2}}{2016} \right)$
	terms, which gives a contradiction if we take
	\[ N > 2 + \log_{\deg h} \left( 2016
		\cdot 5^{\LL(BD\max(T,A))} \right). \]
	For $\deg h = 2$ one can apply the same argument
	replacing $h$ with $h \circ h$.
\end{proof}

\bibliographystyle{hplain}
\bibliography{refs-cyclotomic}

\end{document}